\documentclass[12pt]{article}
\usepackage{amsmath, amsthm, amssymb, dsfont, enumitem, graphicx, mathtools,bbm}
\usepackage[symbol]{footmisc}
\usepackage[margin=1.25in]{geometry}
\usepackage{rotating, xcolor, tkz-euclide,adjustbox, caption}
\usepackage[colorlinks=true, urlcolor=blue]{hyperref}
\usetkzobj{all}
\setenumerate{label=(\roman*),itemsep=0pt}

\newtheorem{theorem}{Theorem}

\newtheorem{conjecture}[theorem]{Conjecture}
\newtheorem{problem}[theorem]{Problem}
\newtheorem{proposition}[theorem]{Proposition}
\newtheorem*{LPR}{Lov\'{a}sz Path Removal Conjecture}
\newtheorem{lemma}[theorem]{Lemma}
\theoremstyle{definition}
\newtheorem{definition}[theorem]{Definition}
\newtheorem{claim}{Claim}[theorem]
\numberwithin{theorem}{section}

\newcommand\A{\mathcal{A}}
\newcommand\rmint{\textrm{int}}

\newcommand\rmend{\textrm{end}}
\newcommand\blfootnote[1]{%
  \begingroup
  \renewcommand\thefootnote{}\footnote{#1}%
  \addtocounter{footnote}{-1}%
  \endgroup
}
\renewcommand{\thefootnote}{\fnsymbol{footnote}}
 
\title{7-Connected Graphs are 4-Ordered}
\author{Rose McCarty, Yan Wang, Xingxing Yu}

\begin{document}
\blfootnote{\textit{Mathematics Subject Classification} (2010): 05C38, 05C40.}

\centerline{\LARGE 7-Connected Graphs are 4-Ordered}
\bigskip\bigskip
\centerline{Rose McCarty\footnote{Now at the Department of Combinatorics and Optimization, University of Waterloo. Partially supported by NSF grant DMS-1265564 through X. Yu}, Yan Wang\footnote{Partially supported by NSF grant DMS-1600738 through X. Yu}, Xingxing Yu\footnote{Partially supported by NSF grants DMS-1265564 and DMS-1600738}}
\bigskip
\centerline{School of Mathematics}
\centerline{Georgia Institute of Technology}
\centerline{Atlanta, GA 30332}
\bigskip

\begin{abstract}
\noindent A graph $G$ is $k$-ordered if for any distinct vertices $v_1, v_2, \ldots, v_k \in V(G)$, it has a cycle through $v_1, v_2, \ldots, v_k$ in order. Let $f(k)$ denote the minimum integer so that every $f(k)$-connected graph is $k$-ordered. The first non-trivial case of determining $f(k)$ is when $k=4$, where the previously best known bounds are $7 \leq f(4) \leq 40$. We prove that in fact $f(4)=7$.
\end{abstract}

\textbf{Keywords:} cycles; connectivity; graph linkages; $k$-ordered graphs

%05C38	Paths and cycles 
%05C40  Connectivity
\section{Introduction}The problem of studying connectivity and cycles through specified vertices originates with the classic result \cite{Dirac} that $k$-connected graphs have a cycle through any set of $k$ vertices for $k \geq 2$. Specifying the order the vertices must appear on the cycle is a generalization introduced in \cite{NgSchultz}. The case $k = 4$ is particularly interesting because it is the first case where the ordered and unordered versions differ. Surveys on the existence of cycles through specified vertices can be found in \cite{Faudree} and \cite{Gould}.

Following the terminology of Faudree \cite{Faudree}, for a positive integer $k$, we say a graph is \textit{$k$-ordered} if, for distinct vertices $v_1, v_2, \ldots, v_k$, the graph has a cycle through $v_1, v_2, \ldots, v_k$ in order. We define $f(k)$ as the smallest positive integer so that every $f(k)$-connected graph is $k$-ordered. Clearly, $f(1)=f(2)=2$ and $f(3)=3$. 
Faudree asks for the determination of $f(4)$ in \cite{Faudree}. Goddard \cite{Goddard} and Mukae \textit{et al.}  \cite{Mukae} have short proofs showing that $4$-connected triangulations of surfaces are $4$-ordered.

The best known upper bounds for $f(4)$, and for $f(k)$ in general, follow from work on linkages. We say that a graph $G$ is $k$-\textit{linked} if for any collection of $k$ pairs of vertices $\{\{s_i,t_i\}: i \in \{1,2,\ldots,k\}\}$, there exists a collection $\{P_i: i\in \{1,2,\ldots, k\}\}$ with the path $P_i$ from $s_i$ to $t_i$ for all $i\in \{1,\ldots, k\}$ such that, for any distinct $i,j \in \{1,2,\ldots, k\}$, no vertex of $P_i$ is an internal vertex of $P_j$.
Let $g(k)$ denote the smallest positive integer so that every $g(k)$-connected graph is $k$-linked. From the definitions, if a graph is $k$-linked then it is also $k$-ordered (by considering the $k$-pairs $\{v_1,v_2\}, \{v_2,v_3\}, \ldots, \{v_k,v_1\}$). 
Thus it follows that $f(k) \leq g(k)$ for every $k$.

Bollob\'{a}s and Thomason \cite{Bollobas-Thomason} were the first to show that $g(k)$ is linear in $k$, with $g(k)\le 22k$. Kawarabayashi, Kostochka, and G. Yu \cite{Kawarabayashi} improved this to $g(k)\le 12k$, and Thomas and Wollan \cite{Thomas-Wollan} showed further that $g(k)\le 10k$. Better bounds on $g(k)$ are known for $k=3$ and large graphs of bounded tree-width, see \cite{3Linked} and \cite{Frohlich}. 

For the general case where $k \geq 4$, the upper bound of $10k$ is currently the best known on both $g(k)$ and $f(k)$ as far as the authors are aware. So in particular the previously best known upper bounds for $k=4$ are $f(4) \leq g(4)\leq 40$.
Ellingham, Plummer, and G. Yu \cite{Ellingham} proved the following result, implying $f(4)\ge 7$.

\begin{theorem}
\label{lowerBound}\cite{Ellingham}
There exists a 6-connected graph $G$ and distinct vertices $v_1, v_2, v_3,$ $v_4 \in V(G)$ so that $G$ does not contain a path through the vertices $v_1, v_2, v_3, v_4$ in order.
\end{theorem}

The main result of this paper is the following which, combined with Theorem~\ref{lowerBound}, shows that $f(4)=7$.

\begin{theorem}
\label{main}
Every 7-connected graph is 4-ordered.
\end{theorem}

Using a precise structure theorem of X. Yu \cite{Yu-Disjoint1, Yu-Disjoint2, Yu-Disjoint3}, Ellingham, Plummer, and G. Yu  \cite{Ellingham} also proved the following. 

\begin{theorem}\cite{Ellingham}\label{ellingham}
For every 7-connected graph $G$ with distinct vertices $v_1, v_2, v_3, v_4 \in V(G)$, the graph $G$ has a path through $v_1, v_2, v_3, v_4$ in order. 
\end{theorem}

Theorem~\ref{main} shows that no additional connectivity is required to guarantee the existence of a cycle through four vertices in order.

\subsection*{Proof Overview}
Suppose that $G$ is a $7$-connected  graph with four distinct vertices $c_0, c_1, c_2, \allowbreak c_3 \in V(G)$ that has no cycle through $c_0, c_1, c_2, c_3$ in order. Our approach is to find vertex-disjoint cycles $C_0$ and $C_2$ and a set $\{z_0, z_1, z_2, z_3\}\subseteq V(G)\setminus (V(C_0)\cup V(C_2))$ so that $c_0, c_1 \in V(C_0)$, $c_2, c_3 \in V(C_2)$, and for $i=0,1,2,3$, the vertex $z_i$ is a neighbor of $c_i$. Define $H$ to be the graph $G-(V(C_0)\cup V(C_2))$. Then $H$ does not have vertex-disjoint paths, with one from $z_0$ to $z_3$ and the other from $z_1$ to $z_2$. 
So, applying Seymour's characterization on the existence of $2$-linkages \cite{seymour}, the graph $H$ is ``almost'' planar and thus has small cuts. 

We will show how to extend these small cuts to a small cut of $G$ in Section \ref{separatingPairs}. We introduce the notion of ``separating pairs'' which are certain subpaths of the cycles $C_0,C_2$. 
The ends of these paths will be used to extend cuts. We prove some lemmas that help us find paths through certain vertices in case cuts do not extend. 

In Section \ref{3Planar} we introduce the notion of ``3-planar graphs'' to explain precisely what we mean by saying that ``$H$ is almost planar''. 
We will state Seymour's characterization of 2-linked graphs, and prove some lemmas on 3-planar graphs. 

In Section \ref{skeletons}, we show that if $G$ is $7$-connected then it has a special structure called a ``$(c_0, c_1, c_2, c_3)$-skeleton''. This structure allows us to assume that $H$ has no $1$-cut separating two of the vertices in $\{z_0, z_1, z_2, z_3\}$ from the rest. Using this notion of skeletons, in Section \ref{2Conn} we show that we can assume $H$ is $2$-connected. In Section \ref{3Conn} we use different techniques to show that $H$ is $3$-connected.

In Section \ref{final} we first prove a discharging lemma on plane graphs. Then we use the fact that $H$ is $3$-connected and $3$-planar, and the lemmas from Section \ref{3Planar}, to show that in fact $H$ is planar, and all neighbors of $C_0$ and $C_2$ are on the boundary of a single face. We then use the discharging lemma to force a special configuration in $H$, which we can use to find the desired cycle in $G$. 
We conclude and give some additional remarks in Section \ref{conclude}.

\subsection*{Notation}
We conclude this section with notation and terminology we need in the rest of the paper. Suppose $G$ is a graph and $P$ is a path in $G$. Then define $\rmend(P)$ to be the set of vertices of smallest degree of $P$. Define $\rmint(P) \coloneqq V(P)\setminus\rmend(P)$. Notice that if $P$ has two or fewer vertices, then $\rmint(P) = \emptyset$. Given a cycle $C$ and an orientation of $C$, for any distinct $u,v\in V(C)$, we use 
$C[u,v]$ to denote the subpath of $C$ from $u$ to $v$ in clockwise order. Let $C(u,v]:=C[u,v]-u$,  $C[u,v):=C[u,v]-v$, and $C(u,v):=C[u,v]-\{u,v\}$. We use similar notation for subpaths of $P$.

\section{Separating Pairs}
\label{separatingPairs}
\begin{definition}
Let $G$ be a graph, let $C$ be a cycle in $G$,  let $v_0,v_1\in V(C)$ be distinct, and let $A\subseteq V(G)\setminus V(C)$. Then a $(v_0, v_1, C, A)$\textit{-separating pair} is a set of paths $\{R_0, R_1\}$ such that there exists an orientation of $C$ so that
\begin{enumerate}
\item for $i = 0,1$, $R_i$ is a subpath of $C[v_i, v_{1-i}]$,
\item for $i=0,1$, $N(A) \cap V(C[v_i, v_{1-i}]) \subseteq V(R_i)$, and
\item the graph $G$ has no edge $uv$ such that $u \in \rmint(R_0)\cup \rmint(R_1)$ and $v \in V(C)\setminus(V(R_0)\cup V(R_1))$. 
\end{enumerate}
A \textit{minimum $(v_0, v_1, C, A)$-separating pair} is a $(v_0, v_1, C, A)$-separating pair $\{R_0, R_1\}$ so that $|V(R_0)|+|V(R_1)|$ is minimum. For all $i,j \in \{0,1\}$, if $R_i\ne \emptyset$  then define $r_i^j$ to be the end of $R_i$ closest to $v_j$ on $C[v_i, v_{1-i}]$.  (Thus, $r_i^0 = r_i^1$ if $R_i$ consists of a single vertex.)
\end{definition}

Clearly, a $(v_0, v_1, C, A)$-separating pair exists as the two paths in $C$ between $v_0$ and $v_1$ form such a pair. Later in the paper we often construct small cutsets containing the set $\rmend(R_0) \cup \rmend(R_1)$. The separating pairs we use will always be chosen to be minimum. The following lemma shows that if $G[V(C)]$ contains no cycle through $v_0$ and $v_1$ that is shorter than $C$, then minimum separating pairs satisfy some additional properties.
\begin{lemma}
\label{mainLemma}
Let $G$ be a graph, let $C$ be a cycle in $G$,  let $v_0,v_1\in V(C)$ be distinct, and let $A\subseteq V(G)\setminus V(C)$. Let $\{R_0, R_1\}$ be a minimum $(v_0, v_1, C, A)$-separating pair. Suppose that $G[V(C)]$ contains no cycle through $v_0$ and $v_1$ with fewer vertices than $C$. Then for each choice $i,j \in \{0,1\}$, if $R_i\ne \emptyset$ then either 
\begin{enumerate}
\item $N(r_i^j) \cap A \neq \emptyset$, or 
\item  $R_{1-i}\ne \emptyset$, $N(r_i^j)\cap \rmint(R_{1-i}) \neq \emptyset$, and $N(r_{1-i}^j) \cap A \neq \emptyset$. 
\end{enumerate}
\end{lemma}
\begin{proof}
For convenience, let $P_0,P_1$ denote the two paths in $C$ between $v_0$ and $v_1$. Without loss of generality, assume that $R_i\subseteq P_i$ for $i=0,1$. 
For $i,j \in \{0,1\}$, if $R_i\ne \emptyset$ let $r_i^j$ be the end of $R_i$ closest to $v_j$ on $P_i$ (with possibly $r_i^0 = r_i^1$). Notice that by \textit{(ii)} of the definition of a separating pair, for $j=0,1$, if $v_j \in N(A)$ then $v_j \in V(R_0)\cap V(R_1)$.

We claim that for  $k,l \in \{0,1\}$ for which $r_k^l$ is defined and $N(r_k^l)\cap A = \emptyset$, we have $R_{1-k}\ne\emptyset$ and $N(r_k^l) \cap \rmint(R_{1-k}) \neq \emptyset$. By symmetry, we assume $r_0^0$ is defined and $N(r_0^0)\cap A = \emptyset$. By the choice of $\{R_0,R_1\}$, $\{R_0-r_0^0, R_1\}$ is not a $(v_0, v_1, C, A)$-separating pair. Hence, there exists $uv \in E(G)$ with $u \in \rmint(R_0-r_0^0)\cup \rmint({R_{1}})$ and $v \in V(C)\setminus(V(R_0-r_0^0)\cup V(R_{1}))$. In particular, $u\in  \rmint(R_0)\cup \rmint(R_1)$. Hence, $v = r_0^0$, since $\{R_0,R_1\}$ is a $(v_0, v_1, C, A)$-separating pair. If $u\in \rmint(R_0-r_0^0)$ then $G[V(C)]$ contains a cycle through $v_0$ and $v_1$ shorter than $C$, a contradiction. So $u \in \rmint(R_{1})$.  

By symmetry, it suffices to prove the assertion for the case $i=0$ and $j=0$. Thus, assume 
that $r_0^0$ is defined and $N(r_0^0)\cap A = \emptyset$. Then by the above claim, $R_{1}\ne\emptyset$ and $N(r_0^0) \cap \rmint(R_{1}) \neq \emptyset$. If $N(r_1^0)\cap A \ne \emptyset$ then we are done. So assume $N(r_1^0)\cap A=\emptyset$. Then by the above claim,  $N(r_1^0) \cap \rmint(R_{0}) \neq \emptyset$. 

Let $u_0\in N(r_1^0)\cap \rmint(R_{0})$ with $P_0[u_0,v_1]$ minimal, $u_1\in N(r_0^0)\cap \rmint(R_{1})$ with $P_1[u_1,v_1]$ minimal, and let $C'$ be the cycle through $v_0$ and $v_1$ defined as follows
$$P_0[v_0,r_0^0]\cup r_0u_1\cup P_1[u_1,v_1]\cup P_0[v_1,u_0]\cup u_0r_1^0\cup P_1[r_1^0,v_0].$$
Then  $V(C') \subseteq V(C)$,  and $V(C)\setminus V(C') = \rmint(P_0[u_0,r_0^0])\cup \rmint(P_{1}[u_{1}, r_{1}^0])$. Hence, $P_0[r_0^0,u_0]=r_0^0u_0$ and $P_1[r_1^0,u_1]=r_1^0u_1$, as $G[V(C)]$ contains no cycle through $v_0$ and $v_1$ and shorter than $C$.  

Let $R_0'=R_0-r_0^0$ and $R_1'=R_1-r_1^0$. Clearly, $R_i'\subseteq P_i$ for $i\in \{0,1\}$, and $N(A)\cap V(C)\subseteq V(R_0')\cup V(R_1')$. By the choice of $u_0$ and $u_1$, we see that $G$ has no edge from $\rmint(R_0')\cup \rmint(R_1')$ to $V(C)\setminus V(R_0'\cup R_1')$. So $\{R_0',R_1'\}$ is a $(v_0, v_1, C, A)$-separating pair, contradicting the choice of $\{R_0,R_1\}$.  
\end{proof}

Now we prove a technical lemma on the existence of several types of paths and cycles in $G[V(C) \cup A]$. The proof is tedious case analysis, but we will use this lemma frequently.

\begin{lemma}
\label{pathExistence}
Let $G$ be a graph, let $C$ be a cycle in $G$, let $v_0,v_1\in V(C)$ be distinct, and let $A\subseteq V(G)\setminus V(C)$. Let $\{R_0, R_1\}$ be a minimum $(v_0, v_1, C, A)$-separating pair and let $u_0 \in \rmint(R_0)$. Suppose that $G[V(C)]$ contains no cycle through $v_0$ and $v_1$ with fewer vertices than $C$. Then 
\begin{enumerate}
\item there exists  $a \in V(C)$ with $N(a) \cap A \neq \emptyset$ such that the graph $G[V(C)]$ contains a path through $u_0, v_0, v_1, a$ in order,
\item if $G[A]$ is connected then $G[V(C) \cup A]-\rmint(R_0)$ contains a cycle though $v_0$ and $v_1$, and
\item if $G[A]$ is connected then for every vertex $u_1 \in \rmint(R_1)$, $G[V(C) \cup A]$ contains a path through $u_0, v_0, v_1, u_1$ in order.
\end{enumerate}
\end{lemma}
\begin{proof}
For convenience, let $P_0$, $P_1$ denote the two paths in $C$ between $v_0$ and $v_1$ containing $R_0,R_1$, respectively. For all $i,j \in \{0,1\}$, if $R_i$ is non-empty then  let $r_i^j$ be the end of $R_i$ closest to $v_j$ on $P_i$. Since we assume $u_0\in \rmint(R_0)$,  $r_0^0$ and $r_0^1$ are defined and distinct. For any distinct $x,y\in V(C)\cap N(A)$, if $G[A]$ is connected then we use $A[x,y]$ to denote a path in $G[A\cup \{x,y\}]$ from $x$ to $y$.

We prove (i) first. If $N(r_0^1)\cap A \neq \emptyset$ then $C-P_0(u_0,r_0^1)$ gives the desired path for (i). So assume 
$N(r_0^1)\cap A = \emptyset$. Then by Lemma~\ref{mainLemma}, $r_1^1\in N(A)$ and there exists $w_1 \in \rmint(R_1)\cap N(r_0^1)$. So 
$(C-P_0(u_0,r_0^1)-P_1(w_1,r_1^1))\cup w_1r_0^1$ gives the desired path for (i). 

To prove (ii), assume $G[A]$ is connected.  If $N(r_0^0)\cap A \neq \emptyset$ and $N(r_0^1)\cap A \neq \emptyset$ then 
 $(C-P_0(u_0,r_0^1))\cup A[r_0^0,r_0^1]$ gives the desired cycle for (ii). So by symmetry, we may assume 
$N(r_0^1)\cap A = \emptyset$. Then by Lemma~\ref{mainLemma}, $N(r_1^1)\cap A \ne  \emptyset$ and  
there exists $w_1 \in \rmint(R_1)\cap N(r_0^1)$. If $N(r_0^0)\cap A \ne \emptyset$ then $(C-\rmint(R_0)-P_1(w_1,r_1^1))\cup 
A[r_0^0,r_1^1]\cup w_1r_0^1$ gives the desired cycle for (ii). So assume $N(r_0^0)\cap A =\emptyset$. Then by  
 Lemma~\ref{mainLemma}, $N(r_1^0)\cap A \ne  \emptyset$ and  
there exists $x_1 \in \rmint(R_1)\cap N(r_0^0)$. Now $(C-\rmint(R_0)-\rmint(R_1))\cup 
A[r_1^0,r_1^1]\cup r_0^0x_1\cup P_1[x_1,w_1]\cup w_1r_0^1$ gives the desired cycle for (ii).

To prove (iii), assume $G[A]$ is connected, and let $u_1 \in \rmint(R_1)$. 
If $N(r_1^0)\cap A \neq \emptyset$ and $N(r_0^1)\cap A \neq \emptyset$ then 
$P_0[u_0, v_0]\cup P_1[v_0, r_1^0]\cup A[r_1^0,r_0^1] \cup P_0[r_0^1, v_1]\cup P_1[v_1, u_1]$ gives the desired path for (iii). 
So we may assume by symmetry that  $N(r_0^1)\cap A =\emptyset$. Then by Lemma~\ref{mainLemma}, $N(r_1^1)\cap V(A) \neq \emptyset$ and there exists a vertex $w_1 \in \rmint(R_1)\cap N(r_0^1)$. If $N(r_1^0)\cap A \ne \emptyset$ 
then $P_0[u_0, v_0]\cup P_1[v_0, r_1^0]\cup A[r_1^0,r_1^1] \cup P_1[r_1^1, v_1]\cup P_0[v_1, r_0^1]\cup r_0^1w_1\cup P_1[w_1,u_1]$ gives the desired path for (iii). So we may assume $N(r_1^0)\cap A = \emptyset$. Then by Lemma~\ref{mainLemma}, 
$N(r_0^0)\cap V(A) \neq \emptyset$ and there exists a vertex $w_0 \in \rmint(R_0)\cap N(r_1^0)$. Now
$P_0[u_0, w_0]\cup w_0r_1^0\cup P_1[r_1^0,v_0]\cup P_0[v_0,r_0^0]\cup A[r_0^0,r_1^1] \cup P_1[r_1^1, v_1]\cup P_0[v_1, r_0^1]\cup r_0^1w_1\cup P_1[w_1,u_1]$ gives the desired path for (iii).
\end{proof}

\section{3-Planar Graphs}
\label{planar}

In this section, we introduce the notion of 3-planar graphs, and state a characterization of 2-linked graphs. 

Let $G$ be a graph and let $s_1, t_1, s_2, t_2 \in V(G)$ be distinct vertices. Then an \textit{$(\{s_1, t_1\}, \{s_2, t_2\})$-linkage} is a set of two disjoint paths $P_1$ and $P_2$ such that for $i = 1,2$, $\rmend{(P_i)}=\{s_i, t_i\}$. To state a result on graphs without an $(\{s_1, t_1\}, \{s_2, t_2\})$-linkage, we use the following notion  due to Seymour \cite{seymour}, which can also  be found in \cite{Yu-Disjoint1}. 

\begin{definition}
A \textit{3-planar} graph $(G, \A)$ consists of a graph $G$ and a family $\A = \{A_1, \ldots, A_k\}$ of pairwise disjoint subsets of $V(G)$ (allowing $\A = \emptyset$) such that
\begin{enumerate}
\item for $1\leq i\neq j \leq k$, $N(A_i)\cap A_j=\emptyset$,
\item for $1 \leq i \leq k$, $|N(A_i)|\leq 3$, and
\item if $p(G, \A)$ denotes the graph obtained from $G$ by (for each $i$) deleting $A_i$ and adding new edges joining every pair of distinct non-adjacent vertices in $N(A_i)$, then $p(G, \A)$ can be drawn in a closed disc $D$ with no pair of edges crossing such that, for each $A_i$ with $|N(A_i)| = 3$, $N(A_i)$ induces a facial triangle in $p(G, \A)$.
\end{enumerate}
If, in addition, $b_1, \ldots, b_n$ are some vertices in $G$ such that $b_i \notin A_j$ for any $A_j \in \A$ and $b_1, \ldots ,b_n$ occur on the boundary of $D$ in that cyclic order, 
then we say that $(G, \A, b_1, \ldots, b_n)$ is 3-planar. We will say that such a drawing is a \textit{plane drawing of} $(G, \A, b_1, \ldots, b_n)$. We will say that $(G, b_1, \ldots, b_n)$ is 3-planar if there exists a collection $\A$ so that $(G, \A, b_1, \ldots, b_n)$ is 3-planar. If $(G, \emptyset, b_1, \ldots, b_n)$ is 3-planar we will say that $(G, b_1, \ldots, b_n)$ \textit{is planar}.
\end{definition}

The main tool we will use is the following theorem due to Seymour \cite{seymour}, while different versions are proved in \cite{Chakravarti, shiloach, Thomassen}.  

\begin{theorem}
\label{2link}
Let $G$ be a graph and let $s_1, t_1, s_2, t_2$ be distinct vertices of $G$. Then $G$ contains no $(\{s_1, t_1\}, \{s_2, t_2\})$-linkage if and only if $(G, s_1, s_2, t_1, t_2)$ is 3-planar.
\end{theorem}

\label{3Planar}
It is convenient for us to develop some lemmas on 3-planar graphs before we begin the proof of our main Theorem \ref{main}. The main lemmas in this section are Lemmas \ref{minimal} and \ref{3connPlanar}, which we will use in later sections to show that a certain 3-planar graph is in fact planar. First we have a definition from \cite{Yu-Disjoint1}.

\begin{definition}
Let $(G, \A)$ be 3-planar, let $A \in \A$ with $N(A) = \{a_1,\ldots,a_m\}$ (where
$m \leq 3$), and let $H = G[A\cup N(A)]$. We say that $A$ is \textit{minimal} if there is no collection $\mathcal{H}$
of pairwise disjoint subsets of $A$ such that $\mathcal{H} \neq \{A\}$ and $(H, \mathcal{H} , a_1,\ldots ,a_m)$ is 3-planar.
We say that $\A$ is minimal if every member of $\A$ is minimal.
\end{definition}

From \cite{Yu-Disjoint1}, we have the following.

\begin{lemma}
\label{minimal}
If $(G, b_0,...,b_n)$ is 3-planar, then there is a collection $\A$ of pairwise disjoint subsets of $V(G)\setminus \{b_0,\ldots,b_n\}$ such that $(G, \A, b_0,\ldots,b_n)$ is 3-planar and $\A$ is minimal.
\end{lemma}

Now we give two propositions to help prove Lemma \ref{3connPlanar} when extending a linkage in $p(G,\A)$ to a linkage in $G$.
\begin{proposition}
\label{3planarityLinkage}
Let $(G, \A)$ be 3-planar so that $\A$ is minimal. Let $s_1, t_1, s_2 \in V(p(G,\A))$ be distinct, and let $t_2 \in V(G)$. Let $t_2^* = t_2$ when $t_2 \in V(p(G, \A))$, and let $t_2^*$ be an arbitrary vertex in $N(A)$ when $t_2 \in A$ for some $A \in \A$. Suppose that $t_2^* \notin \{s_1, t_1, s_2\}$ and $p(G, \A)$ contains an $(\{s_1, t_1\}, \{s_2, t_2^*\})$-linkage. Then $G$ contains an $(\{s_1, t_1\}, \{s_2, t_2\})$-linkage.
\end{proposition}
\begin{proof}
Apply Proposition 3.2 of \cite{Yu-Disjoint1}, with $b \coloneqq s_2$, $b'\coloneqq s_1$, $v \coloneqq t_1$, and $u\coloneqq t_2$. Note that condition (ii) for Proposition 3.2 is guaranteed 
by Proposition 3.1 of \cite{Yu-Disjoint1}. (The proof basically starts with disjoint paths $S_1,S_2$ in $p(G,\A)$ from $s_1,s_2$ to $t_1,t_2^*$, respectively, 
replaces each edge $uv\in E(S_1\cup S_2)-E(G)$ with a path in $G[A\cup N(A)]$ (for some $A\in \A$ with $u,v\in N(A)$), and extends the path from $t_2^*$ to $t_2$.) 
\end{proof}

\begin{proposition}
\label{planarityLinkage}
Fix a plane drawing of a 3-connected planar graph $G$ with outer cycle $Z$. Let $s_1,t_1,s_2 \in V(Z)$ be distinct vertices and let $t_2^* \in V(G)\setminus V(Z)$. Then $G$ has an $(\{s_1,t_1\},\{s_2,t_2^*\})$-linkage.
\end{proposition}
\begin{proof}
Let $P$ denote the path in $Z-s_2$ between $s_1$ and $t_1$. If $G-V(P)$ contains a path $Q$ from $t_2^*$ to $s_2$ then $P,Q$ form the desired $(\{s_1,t_1\},\{s_2,t_2^*\})$-linkage. Hence, we may assume that such $Q$ does not exist. Let $C$ denote the component of $G-V(P)$ containing $t_2^*$. Then $s_2\notin V(C)$. Since $G$ is 3-connected, $|N(C)\cap V(P)|\ge 3$. So let $v_0,v_1\in N(C)\cap V(P)$ with $P[v_0,v_1]$ maximal. Then by planarity of $G$, $\{v_0,v_1\}$ is a cut, a contradiction. 
\end{proof}

Now we are ready to prove the final lemma for this section.

\begin{lemma}
\label{3connPlanar}
Let $(G, \A)$ be 3-planar so that $G$ is 3-connected and $\A$ is minimal. Fix some plane drawing of $p(G, \A)$ and let $F$ be the set of vertices on the outer face. 
Suppose that $|F| \geq 4$. 
Let $s_1, t_1, s_2$ be three distinct vertices in $F$, and let $t_2 \in V(G)\setminus F$. Then $G$ has an $(\{s_1,t_1\},\{s_2, t_2\})$-linkage.
\end{lemma}
\begin{proof}
Since $G$ is 3-connected and $p(G,\A)$ has at least four vertices, $p(G,\A)$ is $3$-connected. Let $Z$ be the outer cycle of the plane drawing of $p(G, \A)$. Then $Z$ is chordless in $p(G,\A)$, and $|V(Z)|=|F|\ge 4$.
If $t_2 \in V(p(G, \A))$, define $t_2^* = t_2$. Otherwise, there is an $A \in \A$ so that $t_2 \in A$. Then define $t_2^*$ to be a vertex in $N(A)\setminus V(Z)$, which  
exists since $|N(A)|=3$, $Z$ is chordless in $p(G,\A)$, and $|V(Z)|\ge 4$.

In either case, by Proposition \ref{planarityLinkage}, $p(G,\A)$ has an $(\{s_1,t_1\},\{s_2,t_2^*\})$-linkage. Then by Proposition \ref{3planarityLinkage}, $G$ has an $(\{s_1,t_1\},\{s_2, t_2\})$-linkage.
\end{proof}

\section{Skeletons}
\label{skeletons}

The objective of this section is to find an intermediate structure in a graph $G$, which we call a ``skeleton'',  that will be helpful for finding an ordered cycle in $G$. 

\begin{definition}
Let $G$ be a graph and let $\{c_0, c_1, c_2, c_3\} \subseteq V(G)$. Then a $(c_0, c_1, c_2, c_3)$-\textit{skeleton} is an ordered list 
$S=(C_0, C_2, Z, P_0, P_1, P_2, P_3)$ such that

\begin{enumerate}
\item  $C_0$, $C_2$, and $Z$ are vertex disjoint cycles, $c_0,c_1\in V(C_0)$,  $c_2,c_3\in V(C_2)$, 

\item  for each $i\in \{0,1,2,3\}$, $P_i$ is a path in $G$ from $c_i$ to some vertex  $z_i\in V(Z)$ that is internally disjoint from 
$V(C_0)\cup V(C_2)\cup V(Z)$, 

\item  $P_0,P_1,P_2,P_3$ are pairwise vertex disjoint, and $z_{0}, z_{1}, z_{3}, z_{2}$ occur on $Z$ in this cyclic order.
\end{enumerate}
\end{definition}

An illustration of a skeleton  is given in Figure \ref{fig:skeleton}. 
We will view $S = C_0\cup C_2\cup Z \cup (\bigcup_{i=0}^3 P_i)$. There are several main lemmas we need on skeletons. 

	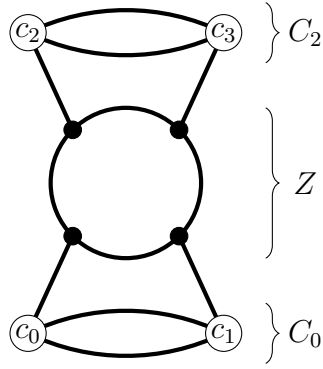
\begin{figure}
		\centering
        \trimbox{0cm 2cm 0cm 3cm}{
		\begin{tikzpicture}[every node/.style={inner sep=.6,outer sep=0}]
			\draw[ultra thick] (0,0) circle (1cm);
            \tkzDefPoint(-1.3,2){C}
        	\tkzDefPoint(0,2.3){B}
            \tkzDefPoint(1.3,2){A}
			\tkzCircumCenter(A,B,C)\tkzGetPoint{O}
			\tkzDrawArc[ultra thick, color=black](O,A)(C)
            \tkzDefPoint(-1.3,2){A}
        	\tkzDefPoint(0,1.7){B}
            \tkzDefPoint(1.3,2){C}
			\tkzCircumCenter(A,B,C)\tkzGetPoint{O}
			\tkzDrawArc[ultra thick, color=black](O,A)(C)
            
            \tkzDefPoint(-1.3,-2){A}
        	\tkzDefPoint(0,-2.3){B}
            \tkzDefPoint(1.3,-2){C}
			\tkzCircumCenter(A,B,C)\tkzGetPoint{O}
			\tkzDrawArc[ultra thick, color=black](O,A)(C)
            \tkzDefPoint(-1.3,-2){C}
        	\tkzDefPoint(0,-1.7){B}
            \tkzDefPoint(1.3,-2){A}
			\tkzCircumCenter(A,B,C)\tkzGetPoint{O}
			\tkzDrawArc[ultra thick, color=black](O,A)(C)
            %THE NODES:
            \node[draw, circle, fill=white] (c1) at (-1.3,-2) {$c_{0}$};
            \node[draw, circle, fill=white] (c4) at (1.3,-2) {$c_{1}$};
            \node[draw, circle, fill=white] (c3) at (-1.3,2) {$c_{2}$};
            \node[draw, circle, fill=white] (c2) at (1.3,2) {$c_{3}$};
            \draw[ultra thick] (c1) -- (225:1cm);
            \draw[ultra thick] (c4) -- (-45:1cm);
            \draw[ultra thick] (c3) -- (135:1cm);
            \draw[ultra thick] (c2) -- (45:1cm);
            \node[draw, circle, inner sep=.2em, fill=black] at (135:1) {};
            \node[draw, circle, inner sep=.2em, fill=black] at (225:1) {};
            \node[draw, circle, inner sep=.2em, fill=black] at (-45:1) {};
            \node[draw, circle, inner sep=.2em, fill=black] at (45:1) {};
            \draw [decorate,decoration={brace,amplitude=5pt},xshift=-4pt,yshift=0pt] (2,2.4) -- (2,1.6) node [black,midway,xshift=15pt] {$C_2$};
            \draw [decorate,decoration={brace,amplitude=5pt},xshift=-4pt,yshift=0pt] (2,-1.6) -- (2,-2.4) node [black,midway,xshift=15pt] {$C_0$};
            \draw [decorate,decoration={brace,amplitude=5pt},xshift=-4pt,yshift=0pt] (2,1) -- (2,-1) node [black,midway,xshift=15pt] {$Z$};
		\end{tikzpicture}
        }
        \caption{A depiction of a $(c_0, c_1, c_2, c_3)$-skeleton.}
        \label{fig:skeleton}
		\end{figure}

\begin{lemma}
\label{anySkeleton} Let $G$ be a graph with distinct vertices $c_0, c_1,c_2, c_3 \in V(G)$ and let $S=(C_0, C_2, Z, P_0, P_1, P_2, P_3)$ be a $(c_0, c_1, c_2, c_3)$-skeleton in $G$. Suppose $G$ has no cycle through $c_0, c_1, c_2, c_3$ in order. Then, for any path $P$ in $G$ from $V(C_0)$ to $V(C_2)$ that is internally disjoint from $S$, either $\rmend(P) = \{c_{0}, c_2\}$ or $\rmend(P) = \{c_1, c_3\}$.
\end{lemma}
\begin{proof} Let $P$ be a path in $G$ from $V(C_0)$ to $V(C_2)$ that is internally disjoint from $S$. 
By symmetry, we may assume that either $c_{0}$ is an end of $P$, or $\rmend(P)\cap \{c_0, c_1, c_2, c_3\} = \emptyset$. 
Then, we may assume $c_{2} \notin \rmend(P)$, as otherwise $\rmend(P)=\{c_0,c_2\}$ and we are done.  

For $i = 0,2$, let $v_i$ be the end of $P$ in $V(C_i)$, and fix an orientation of $C_i$ so that $c_{i+1}$ is not in the path $C_i[v_i, c_{i}]$. 
Fix any orientation of the cycle $Z$. For $j = 1,2$, let $z_j$ be the end of $P_j$ on $Z$. Now
\begin{align*}
C_0[v_0, c_{0}]&\cup C_0[c_{0}, c_{1}] \cup P_1 \cup Z[z_1, z_2]\\
&\cup P_2 \cup C_2[c_{2}, c_{3}] \cup C_2[c_{3}, v_2] \cup P
\end{align*}
is a cycle through $c_0, c_1, c_2, c_3$ in order, a contradiction.
\end{proof}

\begin{lemma}
\label{skeletonExists} Suppose that $G$ is a 7-connected graph with distinct vertices $c_0, c_1,c_2, c_3 \in V(G)$ so that $G$ has no cycle through $c_0, c_1, c_2, c_3$ in order. Then, up to cyclically permuting the labels of the vertices $c_0, c_1, c_2, c_3$, the graph $G$ contains a $(c_0, c_1, c_2, c_3)$-skeleton.
\end{lemma}
\begin{proof}
First, we may assume that 
\begin{itemize}
\item [(1)] $c_ic_{i+1}\notin E(G)$ for $i=0,1,2,3$ (with $c_4=c_0$). 
\end{itemize}
For, suppose (1) fails. Without loss of generality, assume $c_3c_0\in E(G)$. By Theorem~\ref{ellingham}, $G$ has a path $P$ between $c_0$ and $c_3$ such that $c_0,c_1,c_2,c_3$ occur on $P$ in that order. Now $P \cup c_3c_0$ is a cycle through $c_0, c_1, c_2, c_3$ in order, a contradiction.

\medskip

We may also assume that, by cyclically permuting the labels of the  vertices $c_0,c_1,c_2,c_3$ if necessary, 
\begin{itemize}
\item [(2)] there exist a family $B_0$ of three internally disjoint paths from $c_0$ to $c_1$ and a family  $B_2$ of 
three internally disjoint paths from $c_2$ to $c_3$, such that no path in $B_0$ intersects a path in $B_2$. 
\end{itemize}
To see this, we do the following. Let $G'$ be the graph obtained from $G$ by adding two copies of each of $c_0, c_1, c_2$, and  $c_3$. That is, for each $i \in \{0,1,2,3\}$, each copy of $c_i$ has the same neighborhood as $c_i$. Let $S_{0,2}\subseteq V(G')$ be the set consisting of $c_0$ and $c_2$ and all of their copies, and let $S_{1,3}\subseteq V(G')$ be the set consisting of $c_1$ and $c_3$ and all of their copies. So  $|S_{0,2}| = |S_{1,3}| = 6$. 

Since $G$ is 7-connected, $G'$ is 7-connected. Hence, $G'$ contains a set of six pairwise vertex-disjoint paths from $S_{0,2}$ to $S_{1,3}$.  These six paths in $G'$ correspond to six internally disjoint paths in $G$ with ends in $\{c_0,c_1,c_2,c_3\}$, and for $i = 0,1,2,3$, let $A_i$ be the set of all three corresponding paths in $G$ with $c_i$ as an end.

Suppose $A_0 \cap A_1 \neq \emptyset$ and $A_0 \cap A_3 \neq \emptyset$. Then $|A_0\cap A_1| < 3$ and so $A_1\cap A_2 \neq \emptyset$. Likewise $A_3 \cap A_2 \neq \emptyset$. But then the union of one path each from $A_0 \cap A_1$, $A_1 \cap A_2$, $A_2 \cap A_3$, and $A_3 \cap A_0$ is a cycle in $G$ through $c_0, c_1, c_2, c_3$ in order, a contradiction. 

So  $A_0 = A_1$ or $A_0 = A_3$. By symmetry, we may assume $A_0=A_1$ (by relabeling $c_0,c_1,c_2,c_3$ as $c_0, c_3,c_2,c_1$ in the reverse cyclic order). 
Then $A_2=A_3$. Now $A_0, A_2$ give rise to the desired $B_0,B_2$, completing the proof of (2). 

\medskip

We also view $B_0$ (respectively, $B_2$) as a subgraph of $G$ which is the union of the three paths in $B_0$ (respectively, $B_2$).
\begin{itemize}
\item [(3)] There exist vertex-disjoint paths $R_0$ and $R_1$ from $V(B_0)\setminus\{c_0,c_1\}$ to $V(B_2)\setminus\{c_2,c_3\}$ that are internally disjoint from $B_0\cup B_2$.
\end{itemize}
By (1), $|V(B_0)\setminus\{c_0,c_1\}|, |V(B_2)\setminus\{c_2,c_3\}|\geq 3$. Then (3) follows since $G-\{c_0, c_1, c_2, c_3\}$ is $3$-connected.

\medskip

For every $i \in \{0,1\}$ and $j \in \{0,2\}$, let $r_i^j$ be the end of the paths $R_i$ in $V(B_j)$. We claim that

\begin{itemize}
\item [(4)]  $r_0^0$ and $r_1^0$ are on the same path in $B_0$ and the  $r_0^2$ and $r_1^2$ are on the same path in $B_2$.
\end{itemize}
Suppose (4) fails and by symmetry between $B_0$ and $B_2$, we may assume that the vertices $r_0^0$ and $r_1^0$ are on different paths in $B_0$. 
Then $B_0$ contains both a path $P_0$ through $r_0^0, c_0, c_1, r_1^0$ in order, as well as  a path $P_0'$ through $r_0^0, c_1, c_0, r_1^0$ in order. 
Note that  $B_2$ contains either a path $P_2$ through $r_0^2, c_2, c_3, r_1^2$ in order or a path $P_2'$ through $r_0^2, c_3, c_2, r_1^2$ in order. 
Then either $P_2 \cup R_0 \cup P_0'\cup R_1$ or $P_2' \cup R_0 \cup P_0\cup R_1$
is a cycle in $G$ through $c_0, c_1, c_2, c_3$ in order, a contradiction. 

\medskip
For $j = 0,2$, let $P_0^j, P_1^j, P_2^j$ be the paths in $B_j$ and, by (4), assume that the ends of $R_0$ and $R_1$ in $V(B_j)$ are on the path $P_0^j$. 
Relabel the paths $R_0$ and $R_1$ if necessary so that   $c_0, r_0^0, r_1^0, c_1$ occur on $P_0^0$ in that order. 

Note $c_2, r_0^2, r_1^2, c_3$ occur on $P_0^2$ in that order. For, otherwise, the cycle
$P_0^0[r_0^0, c_0]\cup P_1^0\cup P_0^0[c_1, r_1^0]\cup R_1\cup P_0^2[r_1^2, c_2]\cup P_1^2\cup P_0^2[c_3, r_0^2]\cup R_0$ is a cycle in $G$ through  $c_0, c_1, c_2, c_3$ in order, a contradiction.

Let $Z\coloneqq P_0^0[r_0^0, r_1^0]\cup R_1\cup P_0^2[r_1^2, r_0^2]\cup R_0$. Then 
$$S \coloneqq (P_1^0\cup P_2^0, P_1^2\cup P_2^2, Z, P_0^0[c_0, r_0^0], P_0^0[c_1, r_1^0], P_0^2[c_2, r_0^2], P_0^2[c_3, r_1^2])$$ is a $(c_0, c_1, c_2, c_3)$-skeleton.
\end{proof}

\section{Skeletons with Connectivity Properties}
\label{2Conn}
This section is dedicated to proving the following.

\begin{proposition}
\label{propSkeleton}
Suppose that $G$ is a $7$-connected graph, $c_0, c_1, c_2, c_3  \in V(G)$ are distinct, and $G$ has no cycle through $c_0, c_1, c_2, c_3$ in order. 
Then, up to cyclically permuting the labels of the vertices $c_0, c_1, c_2, c_3$, the graph $G$ has a $(c_0, c_1, c_2, c_3)$-skeleton $S=(C_0, C_2, Z, P_0, P_1, P_2, P_3)$ such that
\begin{enumerate}
\item for every $i \in \{0,2\}$, the graph $G[V(C_i)]$ has no cycle through $c_{i}$ and $c_{i+1}$ with fewer vertices than $C_i$,
\item for every $j \in \{0,1,2,3\}$, $|V(P_j)|=2$, and
\item the graph $G-(V(C_0)\cup V(C_2))$ is 2-connected.
\end{enumerate}
\end{proposition}

\begin{proof}
By Lemma \ref{skeletonExists}, some $(c_0, c_1, c_2, c_3)$-skeleton exists. Let $S=(C_0, C_2, Z, \allowbreak P_0, P_1, \allowbreak P_2, P_3)$ be a  $(c_0, c_1, c_2, c_3)$-skeleton,  $H=G-(V(C_0)\cup V(C_2))$, and 
 $B$ be the block of $H$ containing $Z$. Let $B_1, B_2, \ldots B_m$ be the components of $H - V(B)$ with non-empty intersection with $V(S)$, such that $|V(B_1)| \geq |V(B_2)| \geq \ldots \geq |V(B_m)|$. 
Let $A_1, A_2, \ldots A_{n}$ be the components of $H - V(B)$ with empty intersection with $V(S)$, such that $|V(A_1)| \geq |V(A_2)|\geq \allowbreak \ldots \geq |V(A_{n})|$.
We choose $S$ so that 
\begin{itemize}
\item [(1)] the sum $\sum_{i=0}^3|V(P_i)|$ is minimum,
\item [(2)] subject to (1), $|V(B)|$ is maximum,
\item [(3)] subject to (2), $(|V(B_1)|, |V(B_2)|, \allowbreak \ldots, \allowbreak |V(B_m)|)$ is maximal with respect to lexicographic ordering, and
\item [(4)] subject to (3),  $(|V(A_1)|, |V(A_2)|, \ldots, |V(A_{n})|)$ is maximal with respect to lexicographic ordering.
\end{itemize}
For $i = 0,1,2,3$, let $z_i$ be the end of $P_i$ on $Z$. We will show that the skeleton $S$ satisfies conclusions (i), (ii), and (iii). First we show that it satisfies (i).

\begin{claim}
\label{noShorterCycle}
For each $i \in \{0,2\}$, the graph $G[V(C_i)]$ has no cycle through $c_i$ and $c_{i+1}$ with fewer vertices than $C_i$.
\end{claim}
\begin{proof}
By symmetry between $C_0$ and $C_2$, it suffices to consider the case $i=0$. Suppose $C_0'$ is a cycle in $G[V(C_0)]$ through $c_0$ and $c_1$ such that $|V(C_0')|<|V(C_0)|$. 
Then $V(C_0') \subsetneq V(C_0)$. Write $S' \coloneqq (C_0', C_2, Z, \allowbreak P_0, P_1, P_2, P_3)$ and $H' \coloneqq G-(V(C_0')\cup V(C_2))$, and let $B'$ be the block of $H'$ containing $V(Z)$.
Then $S'$ is a $(c_0, c_1, c_2, c_3)$-skeleton and $V(H) \subsetneq V(H')$. 

Note that (1) holds for $S'$. Since $V(B)\subseteq V(B')$, by (2) above,  we have $V(B)=V(B')$. Let $U$ be a component of $G[V(C_0)]-V(C_0')$. If $N(U)\cap V(B_i)\ne \emptyset$ for some $i\in \{1, \ldots, m\}$ 
then $S'$ contradicts the choice of $S$ via (3). Otherwise,  $S'$ contradicts the choice of $S$ via (4). 
\end{proof}

The following is a convenient step to take on the way to proving that conclusions (ii) and (iii) are satisfied.
\begin{claim}
\label{emptyIntersection}
There is no component of $H-V(B)$ with empty intersection with $V(S)$.
\end{claim}
\begin{proof}
Suppose otherwise. Then the component $A_{n}$ exists. Since $B$ is a block of $H$ and $A_n$ is a component of $H-V(B)$, it follows that $A_n$ has no more than one neighbor in $H$. Then since $G$ is $7$-connected, $A_n$ has at least three neighbors on $C_0$ or $C_2$. By symmetry between $C_0$ and $C_2$, we may assume that $|N(A_{n})\cap V(C_0)|\geq 3$. Let $\{R_0, R_1\}$ be a minimum $(c_0, c_1, C_0, V(A_{n}))$-separating pair.

Now define $X \coloneqq \textrm{int}(R_0)\cup \textrm{int}(R_1)\cup A_n$, $T \coloneqq \textrm{end}(R_0)\cup \textrm{end}(R_1)\cup (N(A_n)\cap V(H))$, and $Y \coloneqq V(G)\setminus (X \cup T)$. Then $|Y|\geq 2$ and $|T|\leq 5$. Thus since $G$ is $7$-connected, there exist edges $uv, u'v' \in E(G)$ with $u,u' \in X$, $v,v' \in Y$, $u \neq u'$, and $v \neq v'$. By Lemma \ref{anySkeleton} and since $|N(A_{n})\cap V(C_0)|\geq 3$, we have that $N(A_n)\cap V(C_2)= \emptyset$. Then by part (ii) of the definition of a separating pair, $N(A_n)\cap V(G-H)\subseteq V(R_0\cup R_1)$. Thus we have that $u,u' \in \textrm{int}(R_0)\cup \textrm{int}(R_1)$. Then by Lemma \ref{anySkeleton} and the definition of a separating pair, we have that $v, v' \in V(H)\setminus (X \cup T)$. Now, up to symmetry between $R_0$ and $R_1$, we distinguish the following two cases.

First, suppose that either both $u,u' \in \rmint(R_0)$, or $u \in \rmint(R_0)$ and $v \notin V(B)$. By Claim \ref{noShorterCycle} and part (ii) of Lemma \ref{pathExistence}, $G[V(C_0)\cup V(A_n)]-\textrm{int}(R_0)$ contains a cycle $C_0'$ through $c_0$ and $c_1$. Write $S' \coloneqq (C_0', C_2, Z, P_0,$ $P_1, P_2, P_3)$ and $H' \coloneqq G-(V(C_0')\cup V(C_2))$, and let $B'$ be the block of $H'$ containing $V(Z)$. Then $S'$ is a $(c_0, c_1, c_2, c_3)$-skeleton (so (1) holds for 
$S'$) and  $(V(H)\setminus V(A_n))\cup \textrm{int}(R_0) \subseteq V(H')$.  If both $v,v' \in V(B)$ then both $u,u' \in \rmint(R_0)$ and so $|V(B')|>|V(B)|$. Otherwise, either $u,u' \in \rmint(R_0)$ and $\{v,v'\}\not\subseteq V(B)$, or $u \in \rmint(R_0)$ and $v \notin V(B)$. In either case, one of the components $B_1, B_2, \ldots, B_m, A_1, A_2, \ldots, A_{n-1}$ grows, a contradiction.

Otherwise, without loss of generality, we may assume that $u\in \rmint(R_0)$, $u' \in \rmint(R_1)$, and both $v,v' \in V(B)$. For $i=2,3$, let $z_i'$ be the vertex in $V(P_i)\cap V(B)$ that is closest to $c_i$ on $P_i$. Then since $v$ and $v'$ are distinct and $B$ is 2-connected, $B$ contains disjoint paths $R_2,R_3$ from $\{v,v'\}$ to $z_2',z_3'$, respectively. By Claim \ref{noShorterCycle} and part (iii) of Lemma \ref{pathExistence}, $G[V(C_0)\cup V(A_n)]$ contains paths $P$ and $P'$ between $u$ and $u'$, such that $P$ goes through $u', c_0, c_1, u$ in order and $P'$ goes through 
$u,c_0,c_1,u'$ in order. Then, fixing an arbitrary cyclic order of $C_2$, if $v\in V(R_2)$ and $v'\in V(R_3)$ then  
$$P \cup uv \cup R_2 \cup P_2[z_2', c_2] \cup C_2[c_2, c_3]\cup P_3[z_3',c_3] \cup R_3 \cup v'u'$$is a cycle through $c_0, c_1, c_2, c_3$ in order, a contradiction.
If $v'\in V(R_2)$ and $v\in V(R_3)$ then $$P'\cup u'v' \cup R_2 \cup P_2[z_2', c_2] \cup C_2[c_2, c_3]\cup P_3[z_3', c_3] \cup R_3 \cup vu$$ 
is a cycle through $c_0, c_1, c_2, c_3$ in order, a contradiction.
\end{proof}

We now prove that conclusion (ii) of the proposition holds.

\begin{claim}
\label{edges}
For every $j \in \{0,1,2,3\}$, $|V(P_j)| = 2$.
\end{claim}
\begin{proof}
By symmetry, we prove the case for $j=0$. Suppose for contradiction that $|V(P_0)|>2$. Let $X_0$ be the component of $H - z_0$ containing $\textrm{int}(P_0)$, and let $X_1$ be the component of $H-z_0$ containing $z_1$. 

Suppose that $H-z_0$ has a component $A$ other than $X_0$ and $X_1$. Then since the graph $S-(V(C_0) \cup V(C_2)\cup \{z_0\})$ has two components, one containing $\textrm{int}(P_0)$ and the other containing $z_1$, it follows that $A\cap V(S) = \emptyset$. Since $B$ is 2-connected, $z_0 \in V(B)$, and $ (B-z_0) \cap V(S) \neq \emptyset$, $A \cap B = \emptyset$. Then $A$ is a component of $H-V(B)$ with $V(A) \cap V(S) = \emptyset$, a contradiction to Claim \ref{emptyIntersection}. So $H-z_0$ has no components other than $X_0$ and $X_1$. Fix a cyclic ordering of $Z$ so that $z_0, z_1, z_3, z_2$ occur in that cyclic order.
\smallskip

Case 1. $X_0=X_1$. 

Then let $P$ be a path in $H-\{z_0\}$ of minimum length so that $P$ has one end in $\textrm{int}(P_0)$ and one end in $(V(S)\cap V(H))\setminus V(P_0)$. Let $u$ be the end of $P$ in $\textrm{int}(P_0)$ and let $v$ be the other end of $P$.

Suppose $v \in V(P_3)\cup V(Z(z_1, z_2))$. Then there is a path $R$ with ends $c_0$ and $c_3$ in $(P_0\cup P_3\cup Z(z_1, z_2)\cup P)-z_0$. Then, fixing arbitrary cyclic orderings of $C_0$ and $C_2$, the graph $C_0[c_0, c_1]\cup P_1 \cup Z[z_2, z_1]\cup P_2 \cup C_2[c_2, c_3]\cup R$ is a cycle through $c_0, c_1, c_2, c_3$ in order, a contradiction.

So $v \in (Z[z_2, z_1]\setminus \{z_0\}) \cup \textrm{int}(P_1) \cup \textrm{int}(P_2)$. In this case we will find a $(c_0, c_1, c_2, c_3)$-skeleton $S'=(C_0, C_2, Z', P_0', P_1', P_2', P_3')$ with $\sum_{j=0}^3|V(P_j')|<\sum_{j=0}^3|V(P_j)|$, which is a contradiction to the choice of $S$. If $v \in \textrm{int}(P_2)$, then define $P_0' \coloneqq P_0[c_0, u]$, $P_1' \coloneqq P_1$, $P_2' \coloneqq P_2[c_2, v]$, $P_3' \coloneqq P_3$, and $Z' \coloneqq P_0[u, z_0]\cup Z[z_0, z_2]\cup P_2[z_2, v]\cup P$. If $v \in Z[z_2, z_0)$, then define $P_0' \coloneqq P_0[c_0, u]$, $P_j' \coloneqq P_j$ for $j = 1,2,3$, and $Z' \coloneqq P_0[u, z_0]\cup Z[z_0, v]\cup P$. The remaining cases are similar.

\smallskip
Case 2. $X_0\ne X_1$. 

We have shown that $X_0$ and $X_1$ are the only components of $H-\{z_0\}$, and that they are distinct. Now let $\{R_0, R_1\}$ be a minimum $(c_0, c_1, C_0, V(X_0))$-separating pair. Define $T \coloneqq \textrm{end}(R_0)\cup \textrm{end}(R_1) \cup \{z_0, c_2\}$, $X \coloneqq X_0 \cup \textrm{int}(R_0)\cup \textrm{int}(R_1)$, and $Y \coloneqq V(G)\setminus (X \cup T)$. Then $|T| \leq 6$ and both $X$ and $Y$ are non-empty. So there is an edge $uv \in E(G)$ with $u \in X$ and $v \in Y$.

Suppose that $v \in V(C_2)$. Fix a cyclic ordering of $C_0$ so that $u \notin V(C_0[c_0, c_1])$ and a cyclic ordering of $C_2$ so that $v \notin V(C_2[c_2, c_3))$. Define $R \coloneqq C_0[c_0, c_1]\cup P_1 \cup Z[z_1, z_2]\cup P_2 \cup C_2[c_2, c_3]$. Then $R$ is a path with ends $c_0$ and $c_3$ that contains the vertices $c_0, c_1, c_2, c_3$ in order. Since $G-\textrm{int}(R)$ has a path from $c_0$ to $c_3$ (using $X$ and $uv$), the graph $G$ has a cycle through $c_0, c_1, c_2, c_3$ in order, a contradiction.

By the definition of a separating pair, $v \notin V(C_0)$. Thus $v \in X_1$ and $u \in \textrm{int}(R_0)\cup \textrm{int}(R_1)$. For $i=2,3$, let $u_i$ be the neighbor of $c_i$ on $P_i$. 

We claim that the graph $H-V(X_0)$ contains disjoint paths $R_2,R_3$ from $\{v,z_0\}$ to $u_2,u_3$, respectively. To see this, let $P$ be a minimum-length path in $X_1$ from $v$ to a vertex $p\in V(S) \cap V(X_1)$. If $p \in V(Z(z_0, z_2)) \cup V(P_1) \cup V(P_3)$, then $H-V(X_0)$ contains a $(\{u_2, z_0\}, \{u_3, v\})$-linkage. Otherwise, $p \in V(P_2) \cup V(Z[z_2, z_0))$ and $H-V(X_0)$ contains a $(\{u_2, v\}, \{u_3, z_0\})$-linkage. 

First suppose  $v\in V(R_2)$ and $z_0\in V(R_3)$. Fix a cyclic ordering of $C_0$ so that $u \notin V(C[c_0, c_1])$ and fix an arbitrary cyclic ordering of $C_2$. Then $$C_0[c_0, u]\cup uv \cup R_2 \cup u_2c_2 \cup C_2[c_2, c_3]\cup c_3u_3 \cup R_3 \cup P_0$$is a cycle through $c_0, c_1, c_2, c_3$ in order, a contradiction.

So assume $z_0\in V(R_2)$ and $v\in V(R_3)$. Fix any cyclic ordering of $C_2$. By Claim \ref{noShorterCycle} and part (i) of Lemma \ref{pathExistence}, there is an edge $ax \in E(G)$ with $x \in X_0$ so that the graph $G[V(C_0)]$ contains a path $P$ with ends $u$ and $a$ through $u, c_0, c_1, a$ in order. Let $P'$ be a path in $G[V(X_0) \cup \{z_0\}]$ with ends $x$ and $z_0$. Then $$P\cup ax\cup P' \cup R_2\cup u_2c_2\cup C_2[c_2, c_3]\cup c_3u_3\cup R_3\cup vu$$ is a path through $c_0, c_1, c_2, c_3$ in order. This is a contradiction, and completes the proof of the claim.
\end{proof}

Now we are ready to finish the proof of Proposition \ref{propSkeleton}. Parts (i) and (ii) hold by Claims \ref{noShorterCycle} and \ref{edges}, respectively. 
Also by Claim \ref{edges}, $V(S) \cap V(H) \subseteq V(B)$. Then every component of $H-V(B)$ has empty intersection with $V(S)$. So by Claim \ref{emptyIntersection}, $V(H) = V(B)$ and thus $H$ is 2-connected; so (iii) holds.
\end{proof}

\section{Raising the Connectivity}
\label{3Conn}
In this section we strengthen conclusion (iii) of Proposition \ref{propSkeleton} so that we will be able to apply Lemma \ref{3connPlanar}. That is, we prove the following:

\begin{proposition}
\label{propSkeleton3Conn}
Suppose that $G$ is a $7$-connected graph with distinct vertices $c_0,$ $c_1, c_2, c_3  \in V(G)$ so that $G$ has no cycle through $c_0, c_1, c_2, c_3$ in order. Then, up to cyclically permuting the labels of the vertices $c_0, c_1, c_2, c_3$, the graph $G$ has a $(c_0, c_1, c_2, c_3)$-skeleton $S=(C_0, C_2, Z, P_0, P_1, P_2, P_3)$ so that:
\begin{enumerate}
\item for every $i \in \{0,2\}$, the graph $G[V(C_i)]$ has no cycle through $c_{0}$ and $c_{1}$ with fewer vertices than $C_i$,
\item for every $j \in \{0,1,2,3\}$, $|V(P_j)|=2$, and
\item the graph $G-(V(C_0)\cup V(C_2))$ is 3-connected.
\end{enumerate}
\end{proposition}
\begin{proof}
Let $S=(C_0, C_2, Z, P_0, P_1, P_2, P_3)$ be a $(c_0, c_1, c_2, c_3)$-skeleton as in Proposition \ref{propSkeleton}. For every $i \in \{0,1,2,3\}$, let $z_i$ be the end of $P_i$ on $Z$. Define $H \coloneqq G-(V(C_0)\cup V(C_2))$. The only thing we need to prove is that $H$ is $3$-connected. Suppose $H$ is not $3$-connected. Let $T\subseteq V(H)$ and $A$ be a component of $H-T$ so that $|T| \leq 2$, $H-T$ is not connected, and $|V(A) \cap \{z_0, z_1, z_2, z_3\}|$ is minimum. Since $G$ is $7$-connected, there exists $i \in \{0,2\}$ so that $N(A) \cap (V(C_i)\setminus \{c_i, c_{i+1}\})$ is non-empty. So by symmetry, we may assume that either $A \cap \{z_0, z_1, z_2, z_3\} = \emptyset$ and $N(A) \cap (V(C_0)\setminus \{c_0, c_{1}\})$ is non-empty, or $z_0 \in V(A)$. First we prove two claims.

\begin{claim}
$z_2,z_3 \notin V(A)$.
\end{claim}
\begin{proof}
Suppose otherwise. Then $V(A) \cap \{z_0, z_1, z_2, z_3\}\ne \emptyset$, and so $z_0 \in V(A)$. Then by the choice of $A$ and since $H-T$ is disconnected, $|V(A) \cap \{z_0, z_1, z_2, z_3\}|=2$. 
Indeed, $V(A) \cap \{z_0, z_1, z_2, z_3\}=\{z_0, z_2\}$. For, otherwise,  $V(A) \cap \{z_0, z_1, z_2, z_3\}=\{z_0, z_3\}$. Then, since $A$ is connected and $H-V(A)$ is connected,  
$H$ has a $(\{z_0, z_3\}, \{z_1, z_2\})$-linkage; so $G$ has a cycle through $c_0, c_1, c_2, c_3$ in order, a contradiction. 

For $i=0,2$, let $\{R_0^i, R_1^i\}$ be a minimum $(c_i, c_{i+1}, C_i, V(A))$-separating pair. Define $$T'\coloneqq (T\cup \textrm{end}(R_0^0)\cup \textrm{end}(R_1^0) \cup \textrm{end}(R_0^2) \cup \textrm{end}(R_1^2))\setminus \{c_0, c_2\}.$$ 
By (ii) of the definition of separating pairs, for $i=0,2$ the vertex $c_i$ is an end of both $R_0^i$ and $R_1^i$. So $|T'|\leq 6$. Define 
$$X \coloneqq \textrm{int}(R_0^0)\cup \textrm{int}(R_1^0) \cup \textrm{int}(R_0^2) \cup \textrm{int}(R_1^2)\cup V(A) \cup \{c_0, c_2\}$$ and let $Y \coloneqq V(G)\setminus (X \cup T')$. 

Since $G$ is $7$-connected, $T'$ is not a cut separating $X$ and $Y$.
So there exists an edge $uv \in E(G)$ with $u \in X$ and $v \in Y$. From (ii) of the definition of a separating pair, we have that $u \notin V(A)$. So by symmetry between $C_0$ and $C_2$, we may assume that $u \in \textrm{int}(R_0^0)\cup \textrm{int}(R_1^0)\cup \{c_0\}$. Then by (iii) of the definition of a separating pair and Lemma \ref{anySkeleton}, we have that $v \in V(H)\setminus (V(A) \cup T)$. Recall that $G[V(C_0)]$ contains no cycle through $c_0$ and $c_1$ with fewer vertices than $C_0$. Then by (i) of Lemma \ref{pathExistence}, there exists an edge $aa' \in E(G)$ with $a \in V(C_0)$ and $a' \in V(A)$ so that $G[V(C_0)]$ contains a path $P$ with ends $a$ and $u$ which goes through the vertices $u,c_0, c_1, a$ in order.

Since $A$ and $H-V(A)$ are both connected, the graph $H$ has a $(\{z_2, a'\},\allowbreak \{z_3, v\})$-linkage $S_2, S_3$ so that $z_2$ is an end of $S_2$ and $z_3$ is an end of $S_3$. Then, fixing an arbitrary cyclic ordering of $C_2$, $$P \cup aa'\cup S_2 \cup z_2c_2 \cup C_2[c_2, c_3] \cup c_3z_3 \cup S_3 \cup vu$$is a cycle through $c_0, c_1, c_2, c_3$ in order, a contradiction.
\end{proof}

Now, let $\{R_0, R_1\}$ be a minimum $(c_0, c_1, C_0, V(A))$-separating pair in $G-(T\cup C_2)$. Next, we show another claim.

\begin{claim}
There is an edge $uv \in V(G)$ with $u \in \textrm{int}(R_0)\cup \textrm{int}(R_1)$ and $v \in V(H)\setminus (T \cup V(A))$.
\end{claim}
\begin{proof}
Define $W \coloneqq \{c_{i+2}: z_i\in V(A)\}$. So $W \subseteq \{c_2, c_3\}$ by Claim 6.1.1. Define $T' \coloneqq T \cup \textrm{end}(R_0)\cup \textrm{end}(R_1)\cup W$, 
$X \coloneqq \textrm{int}(R_0)\cup \textrm{int}(R_1) \cup V(A)$, and $Y \coloneqq V(G)\setminus (T' \cup X)$. 
Then $|T'|\leq 6$ since, for $i=0,1$, if $z_{i} \in V(A)$ then $c_i$ is an end of both $R_0$ and $R_1$. Furthermore, both $X$ and $Y$ are non-empty and $V(G)$ is the disjoint union of $X$, $Y$, and $T'$. Since $G$ is not $7$-connected, $T'$ is not a cut separating $X$ and $Y$. So there is an edge $uv \in E(G)$ with $u \in X$ and $v \in Y$.

If $v \in V(H)\setminus (T \cup V(A))$, then $u \in \textrm{int}(R_0)\cup \textrm{int}(R_1)$ and we are done. By the definition of a separating pair, $v \notin V(C_0)\setminus (V(R_0)\cup V(R_1))$. So $v \in V(C_2)\setminus W$, and by Lemma \ref{anySkeleton}, $u \in V(A)$. 

First suppose $V(A)\cap \{z_0, z_1, z_2, z_3\}=\emptyset$. Then since $H-V(A)$ is connected, it has (not necessarily disjoint) paths $Q_0$ with ends $z_0$ and $z_3$ and $Q_1$ with ends $z_1$ and $z_2$. 
Recall that we assumed  $(N(A) \cap V(C_0))\setminus \{c_0, c_1\}\ne \emptyset$. So the graph $G[V(A)\cup V(C_0)]$ contains (not necessarily disjoint) paths $S_0$ with ends $u$ and $c_0$ containing $c_1$, and $S_1$ with ends $u$ and $c_1$ containing $c_0$. Then $S_0 \cup c_0z_0 \cup Q_0$ is a path  in $G-C_2$ that goes through $u, c_1, c_0, z_3$ in order. Similarly $S_1 \cup c_1z_1 \cup Q_1$ is a path contained in $V(G)\setminus V(C_2)$ that goes through $u, c_0, c_1, z_2$ in order. Hence, $C_2 \cup S_0 \cup c_0z_0 \cup Q_0 \cup uv$ or  $C_2 \cup S_1 \cup c_1z_1 \cup Q_1 \cup uv$ contains a cycle through $c_0, c_1, c_2, c_3$ in order in $G$, a contradiction. (When 
$v\in \{c_2,c_3\}$, only one of these works.)

Now suppose that $V(A)\cap \{z_0, z_1, z_2, z_3\}=\{z_0\}$. Then $c_2\in W$ and $v\ne c_2$. Fix a cyclic ordering of $C_2$ so that $c_3\in V(C[c_2,v])$. Let $S_0$ be a path contained in $A$ with ends $u$ and $z_0$, and let $S_1$ be a path contained in $H-A$ with ends $z_1$ and $z_2$. Fix an arbitrary cyclic ordering of $C_0$. 
Then $$C_0[c_0, c_1]\cup c_1z_1 \cup S_1 \cup z_2c_2 \cup C[c_2, v]\cup vu \cup S_0\cup z_0c_0$$ 
is a cycle through $c_0, c_1, c_2, c_3$ in order, a contradiction.

Finally, let $V(A)\cap \{z_0, z_1, z_2, z_3\}=\{z_0, z_1\}$. Then $W=\{c_2,c_3\}$; so  $v \in V(C_2)\setminus \{c_2, c_3\}$. In this case it suffices to show that $H$ has either a $(\{z_0, u\}, \{z_1, z_2\})$-linkage or a $(\{z_0, z_3\}, \{z_1, u\})$-linkage, as such a linkage, $S$, and $uv$ give a cycle through $c_0,c_1,c_2,c_3$ in order, a contradiction. 
Fix a cyclic ordering of $Z$ so that the vertices $z_0, z_1, z_3, z_2$ occur on $Z$ in that order. Since $V(A)\cap \{z_0, z_1, z_2, z_3\}=\{z_0, z_1\}$, $T$ must contain exactly one vertex in $Z(z_1, z_3]$ and one vertex in $Z[z_2, z_0)$. Let $P$ be a shortest path in $A$ from $u$ to a vertex $x\in V(Z) \cap V(A)$, and let $x$ be the end of $P$ on $Z$. Then $x \in V(Z(z_2, z_3))$. If $x \in V(Z(z_2, z_0])$, then $P \cup Z[x, z_0], Z[z_1,z_2]$ form a $(\{z_0, u\}, \{z_1, z_2\})$-linkage. If $x \in V(Z(z_0, z_3))$, then $Z[z_3, z_0], P \cup Z[x,z_1]$ form a $(\{z_0, z_3\}, \{z_1, u\})$-linkage. This completes the proof of the claim.
\end{proof}

Now we show that there exists $t \in T$ so that $H-V(A)$ contains disjoint paths $Q_2,Q_3$ from $\{t,v\}$ to $z_2,z_3$, respectively. Otherwise, by Menger's Theorem, there is a separation $(X, Y)$ of $H-V(A)$ of order one or less so that $\{z_2, z_3\} \subseteq X$ and $\{v\}\cup T \subseteq Y$. Then $|X\setminus Y|\geq 1$ and $|(Y \cup V(A)) \setminus X| \geq |V(A)| \geq 1$. So $(X, Y \cup V(A))$ is a non-trivial separation of $H$ of order one or less, a contradiction since $H$ is $2$-connected.

By (i) of Lemma \ref{pathExistence}, for $i=0,1$ there is an edge $a_ia_i' \in E(G)$ so that $a_i \in V(C_0)$, $a_i' \in V(A)$, and $G[V(C_0)]$ contains a path $S_i$ with ends $u$ and $a_i$ going through $u, c_{1-i}, c_i, a_i$ in order. For $i=0,1$, let $Q_i$ be a path in $A$ with ends $a_i'$ and $t$.

Fix any cyclic ordering of $C_2$. If $v\in V(Q_2)$ and $t\in V(Q_3)$, then $$Q_2\cup z_2c_2 \cup C_2[c_2, c_3]\cup c_3z_3 \cup Q_3 \cup Q_0 \cup a_0'a_0 \cup S_0\cup uv$$is a cycle through $c_0, c_1, c_2, c_3$ in order, 
a contradiction. 
If $t\in V(Q_2)$ and $v\in V(Q_3)$, then  $$Q_3\cup z_3c_3 \cup C_2[c_2, c_3]\cup c_2z_2 \cup Q_2 \cup Q_1 \cup a_1'a_1 \cup S_1\cup uv$$ 
is a cycle through $c_0, c_1, c_2, c_3$ in order, a contradiction. 

\end{proof}

\section{Discharging and Proof of Theorem \ref{main}}
\label{final}
First we prove a discharging lemma on planar graphs, and then we complete the proof of the main theorem. 

\begin{lemma}
\label{discharging}
Let $H$ be a $3$-connected planar graph with some fixed planar drawing. Let $Z$ be the outer cycle of $H$ and let $x$ and $y$ be distinct vertices in $V(Z)$. Then either
\begin{enumerate}
\item there exists $v \in V(H)\setminus V(Z)$ with $d(v)\leq 6$, or
\item there exists $uv \in E(Z)$ so that $\{u,v\}\cap \{x,y\} = \emptyset$ and $d(u)+d(v)\leq 7$.
\end{enumerate}
\end{lemma}
\begin{proof}
Let $\mathcal{F}$ denote the set of all facial cycles of $G$. Then $2|E(H)|=\sum_{f \in \mathcal{F}}|V(f)|\geq |V(Z)|+3(|\mathcal{F}|-1)$. By Euler's Formula, we have $12-6|V(H)|+6|E(H)|=6|\mathcal{F}|\leq 4|E(H)|+6-2|V(Z)|$. Thus, 
\begin{align*}
2|E(H)|\leq 4|V(H)|+2|V(H)\setminus V(Z)|-6.
\end{align*}
Now for every vertex $v \in V(Z)$, define $ch_0(v) \coloneqq d(v)-4$ and for every $v \in V(H)\setminus V(Z)$, define $ch_0(v) \coloneqq d(v)-7$. Then by the last inequality and since $x$ and $y$ have degrees at least three,
\begin{align*}
\sum_{v \in V(H)\setminus \{x,y\}}ch_0(v)&\leq 2+2|E(H)|-4|V(H)|-3|V(H)\setminus V(Z)|\\&\leq -|V(H)\setminus V(Z)|-4.
\end{align*}
Now, for all distinct vertices $v,u \in V(Z)\setminus \{x,y\}$ so that $uv \in E(Z)$ and $d(u) \geq 6$, give one unit of charge from $u$ to $v$. For all distinct vertices $v,u \in V(Z)\setminus \{x,y\}$ so that $uv \in E(Z)$ and $d(u) = 5$, give $1/2$ unit of charge from $u$ to $v$. Denote the charge function obtained from $ch_0$ this way by $ch$. Observe that $V(H)\setminus V(Z)\ne \emptyset$ as $H$ is $3$-connected, and hence

\begin{align*}
\sum_{v \in V(H)\setminus \{x,y\}}ch(v)=\sum_{v \in V(H)\setminus \{x,y\}}ch_0(v)< -4.
\end{align*}
Suppose that conclusion (i) does not hold. Then for every vertex $v \in V(H)\setminus V(Z)$, $ch(v)=ch_0(v)=d(v)-7 \geq 0$. For every vertex $v \in V(Z)\setminus \{x,y\}$ with $d(v)\geq 6$, $ch(v)\geq ch_0(v)-2 = d(v)-6 \geq 0$. Likewise every vertex $v \in V(Z)\setminus \{x,y\}$ of degree four or five has $ch(v)\geq 0$. If $v \in V(Z)\setminus \{x,y\}$ with degree less than four, then $d(v)=3$ and $ch(v)\geq -1$.

Therefore,  since $\sum_{v \in V(H)\setminus \{x,y\}}ch(v) <-4$, there exists a vertex $v \in V(Z)\setminus \{x,y\}$ with negative charge so that neither of its neighbors on $Z$ are $x$ or $y$. Let $u$ and $u'$ be the neighbors of $v$ on $Z$. Then since $ch(v)<0$, $v$ has degree three and  $u$ or $u'$ has degree no more than four. Thus case (ii) of the lemma holds.
\end{proof}

Now we are ready to complete the proof of Theorem \ref{main}.

\begin{proof} 

Let $G$ be a $7$-connected graph with distinct vertices $c_0, c_1, c_2, c_3 \subseteq V(G)$ and, going for a contradiction, suppose that $G$ has no cycle through $c_0, c_1, c_2, c_3$ in order. Let $S = (C_0, C_2, Z, P_0, P_1, P_2, P_3)$ be a $(c_0, c_1, c_2, c_3)$-skeleton as in Proposition \ref{propSkeleton3Conn}. For every $i \in \{0,1,2,3\}$, let $z_i$ be the end of $P_i$ on $Z$. Define $H \coloneqq G-(V(C_0)\cup V(C_2))$. First we prove a claim.

\begin{claim}
\label{planarDrawing}
There is a plane drawing of the graph $(H, z_0, z_1, z_3, z_2)$ with outer cycle $Z'$ such that $N(C_0\cup C_2)\subseteq V(Z')$.
\end{claim}
\begin{proof}
First of all, since $G$ has no cycle through $c_0, c_1, c_2, c_3$ in order, the graph $H$ has no $(\{z_0, z_3\}, \{z_1, z_2\})$-linkage. So by Theorem \ref{2link}, we have that $(H, z_0, z_1, z_3, z_2)$ is $3$-planar. Then by Lemma \ref{minimal}, there is a collection $\A$ of pairwise disjoint subsets of $V(H)\setminus \{z_0,z_1,z_2,z_3\}$ so that $(H, \A, z_0, z_1, z_3, z_2)$ is 3-planar and $\A$ is minimal. Fix a plane drawing of $p(H, \A)$, and let $F$ be the set of vertices on its outer face (of $p(H, \A)$).

We claim that $N(C_0\cup C_2)\subseteq F$. If this is true, then since $G$ is $7$-connected and for every $A \in \A$, $|N_H(A)|=3$, we will have $\A = \emptyset$ and be done. So suppose that there exists 
$u \in N(C_0\cup C_2)\setminus F$. By symmetry between $C_0$ and $C_2$, we may assume that $u \in N(C_0)$. By symmetry between $c_0$ and $c_1$ on $C_0$, we may assume that there is a vertex $v \in V(C_0)\setminus \{c_1\}$ so that $vu \in E(G)$.

By Lemma \ref{3connPlanar} applied to $(H, \A, z_0, z_1, z_3, z_2)$, the graph $H$ has a $(\{z_1,z_2\},\allowbreak \{z_3, u\})$-linkage $R_1, R_3$, where $R_1$ has ends $z_1$ and $z_2$ and $R_3$ has ends $z_3$ and $u$. Fix any cyclic order of $C_2$, and a cyclic order of $C_0$ so that $v \notin \textrm{int}(C_0[c_0, c_1])$. Then $$C_0[c_0, c_1]\cup c_1z_1 \cup R_1 \cup z_2c_2 \cup C_2[c_2, c_3] \cup c_3z_3 \cup R_3 \cup uv \cup C_0[v, c_0]$$is a cycle through $c_0, c_1, c_2, c_3$ in order, a contradiction.
\end{proof}

\smallskip

Fix a plane drawing of $(H,z_0,z_1,z_3,z_2)$ as in the claim, and fix a cyclic ordering of $Z'$ so that $z_0, z_1, z_3, z_2$ occur on $Z'$ in that order. We have one more claim:

\begin{claim}
\label{betterPlanar}
There exist vertices $x_0,x_1 \in V(Z')$ so that $N(C_0)\cap V(H) \subseteq Z'[x_0,x_1]$ and $N(C_2)\cap V(H)\subseteq Z'[x_1,x_0]$. (So $x_0,z_0,z_1,x_1,z_3,z_2$ occur on $Z'$ in order.)
\end{claim}
\begin{proof}
We may assume that $N(c_0) \cap Z'(z_1,z_2) = \emptyset$ and $N(c_1) \cap Z'(z_3,z_0) = \emptyset$; for, otherwise, $G[V(C_0\cup C_2\cup Z')]$ contains a cycle through $c_0,c_1,c_2,c_3$ in order.
Let $x_0$ be the neighbor of $c_0$ on $Z'$ so that $Z'[z_2, x_0]$ is shortest possible. Let $x_1$ be the neighbor of $c_1$ on $Z'$ so that $Z'[x_1, z_3]$ is shortest possible. 
We now show that $x_0,x_1$ satisfy Claim ~\ref{betterPlanar}. 

First suppose that there is a vertex $u \in Z'(x_1, x_0)$ with a neighbor $v \in V(C_0)$. By symmetry between $c_0$ and $c_1$, we may assume that $v\ne c_1$. If $u \in Z'(x_1, z_2)$, the graph $H$ has a $(\{z_1, z_2\}, \{u, z_3\})$-linkage and thus a cycle through $c_0, c_1, c_2, c_3$ in order, a contradiction. So we may assume that $u \in Z'[z_2, x_0)$. By the choice of $x_0$, $v\ne c_0$. 
Then $H$ has a $(\{z_0, z_3\}, \{u, z_2\})$-linkage, and thus a cycle through $c_0, c_1, c_2, c_3$ in order.

Now suppose that there is a vertex $u \in Z'(x_0, x_1)$ with a neighbor $v \in V(C_2)$. By symmetry between $c_2$ and $c_3$ we may assume that $v \ne c_3$. Then $H$ contains a $(\{x_0, z_3\}, \{u, z_1\})$-linkage, and thus $G$ contains a cycle through $c_0, c_1, c_2, c_3$ in order, a contradiction.
\end{proof}

Since $G$ is $7$-connected and by Claim \ref{betterPlanar}, every vertex in $V(H)\setminus V(Z')$ has degree at least seven in $H$. Thus by Lemma \ref{discharging}, there exists $uv \in E(Z')$ 
so that $\{u,v\}\cap \{x_0, x_1\}=\emptyset$ and $d_H(u)+d_H(v)\leq 7$. By symmetry between the cycles $C_0$ and $C_2$, we may assume that both $u$ and $v$ are vertices in $Z'(x_0, x_1)$, and thus only have neighbors in $V(H) \cup V(C_0)$. 

Let $\{R_0, R_1\}$ be a minimum $(c_0, c_1, C_0, \{u,v\})$-separating pair. Define $A \coloneqq \textrm{int}(R_0)\cup \textrm{int}(R_1)$ and $T \coloneqq \textrm{end}(R_0)\cup \textrm{end}(R_1) \cup \{u,v\}$. 

We may assume that $A\ne \emptyset$. Suppose otherwise, since $d_{C_0}(u)+d_{C_0}(v) \geq 7$ and $|V(C_0)\cap T|\leq 4$, there exists $i \in \{0,1\}$ so that both ends of $R_i$ are adjacent to both $u$ and $v$. 
Since $H$ either contains a $(\{u, z_3\}, \{v, z_2\})$-linkage or a $(\{v, z_3\}, \{u, z_2\})$-linkage, the graph $G$ contains a cycle through $c_0, c_1, c_2, c_3$ in order, a contradiction.

Now since $G$ is $7$-connected and $|T| \leq 6$, there exists an edge $xy \in E(G)$ so that $x \in A$ and $y \in V(G)\setminus (A \cup T)$. Since $\{R_0, R_1\}$ is a separating pair, $y \notin V(C_0)$. By Lemma \ref{anySkeleton} and since $x \in V(C_0)\setminus \{c_0, c_1\}$, we have $y \notin V(C_2)$. So $y \in V(H)\setminus \{u,v\}$.

Since $H$ is $3$-connected, $H-v$ is 2-connected and, hence, has two disjoint paths from $\{u, y\}$ to $\{z_2, z_3\}$. By (i) of Proposition \ref{propSkeleton3Conn} and (i) of Lemma \ref{pathExistence}, and since $uv \in E(G)$, for every $i \in \{0,1\}$ the graph $G[V(C_0)\cup \{u,v\}]$ has a path through $u, c_i, c_{1-i}, x$ in order. Thus $G$ has a cycle through $c_0, c_1, c_2, c_3$ in order. 
\end{proof}

\section{Concluding Remarks}
\label{conclude}
Recall that $f(k)$ is the minimum connectivity for a graph to be $k$-ordered, and $g(k)$ is the minimum connectivity for a graph to be $k$-linked. 
Kostochka and G. Yu \cite{KostochkaYu} asked the following.

\begin{problem}
Is it true that $f(k) < g(k)$ for all $k \geq 2$?
\end{problem}
It is not hard to show that $f(2)=2$ and $f(3) = 3$ as there is only one cyclic ordering of three or fewer vertices. Jung showed that $g(2)=6$ \cite{Jung}. It follows that $f(2)<g(2)$ and $f(3)<g(2) \leq g(3)$. As observed in the literature, the graph obtained from the complete graph on $3k-1$ vertices by removing a matching of size $k$ is not $k$-linked \cite{3Linked}. Thus as a corollary of our main Theorem \ref{main}, we have the next case that $f(4) = 7 < 10 \leq g(4)$. 

We also ask for a structural characterization when a graph $G$ with four fixed vertices $v_1, v_2, v_3, v_4$ has no cycle through $v_1, v_2, v_3, v_4$ in order. By this we mean something similar to the Two Paths Theorem 
 \cite{seymour} and \cite{Thomassen}, which we rely on in this paper and introduced in Section 3. The theorem of X. Yu  \cite{Yu-Disjoint1, Yu-Disjoint2, Yu-Disjoint3} characterizing when a graph has a path through four given vertices in a specific order
also motivates our approach. We hope that some of the work in this paper can be used towards finding such a characterization. Much of the structure we expect to see appears in our proof  (see Claim \ref{planarDrawing}).

The techniques we use to prove Proposition \ref{propSkeleton} build upon work on the following:

\begin{LPR}
\label{LPR} \cite{Lovasz}
For every positive integer $k$, there is an integer $h(k)$ so that for every $h(k)$-connected graph $G$ and all vertices $s$ and $t$ in $G$, there is an induced path $P$ with ends $s$ and $t$ so that the graph $G-V(P)$ is $k$-connected.
\end{LPR}

Kawarabayashi and Ozeki \cite{KO} made the following related conjecture.  

\begin{conjecture}\cite{KO}
\label{KO-conjecture} There exists a function $f(k,l)$ such that the following holds. For every $f(k,l)$-connected graph $G$ and two distinct vertices $s$ and $t$ in $G$, there are $k$ internally disjoint paths $P_1,\ldots, P_k$ with endpoints $s$ and $t$ such that $G-\bigcup_{i=1}^k V(P_i)$ is $l$-connected.
\end{conjecture}

The above conjecture is implied by the Lov\'{a}sz Path Removal Conjecture. This can be seen by making copies of $s$ and $t$ (where all copies of $s$ have the same neighborhood as $s$, and likewise for $t$) and repeatedly finding an induced path $P$ between a copy of $s$ and a copy of $t$ that is internally disjoint from all copies of $s$ and $t$. Kawarabayashi and Ozeki \cite{KO} proved that $f(k,1)\le 2k+1$ and $f(k,2)\le 3k+1$. Furthermore, J. Ma proved that with more connectivity a stronger conclusion holds \cite{JieMa}. Our proof of Proposition \ref{propSkeleton} uses similar techniques to \cite{KO} and \cite{JieMa} for the case of $f(2,2)$.

\section*{Acknowledgments}
The authors would like to thank Runrun Liu and the anonymous referees for helpful comments throughout, especially in corrections to Section~\ref{separatingPairs}.

\let\OLDthebibliography\thebibliography
\renewcommand\thebibliography[1]{
  \OLDthebibliography{#1}
  \setlength{\parskip}{0em}
  \setlength{\itemsep}{0pt plus 0.3ex}
}

\end{document}